\documentclass[reqno,a4paper,11pt]{amsart}
\usepackage{amssymb}
\usepackage[dvips, lmargin=3.5cm, rmargin=3.5cm, tmargin=4cm, bmargin=4cm]{geometry}
\usepackage[colorlinks=true, citecolor=red, linkcolor=blue, urlcolor=blue]{hyperref}

\usepackage{xfrac}
\usepackage{graphicx}
\usepackage{hyperref}
\usepackage{amsmath}
\usepackage[T1]{fontenc}
\usepackage{enumitem}
\usepackage{dsfont}
\everymath{\displaystyle}
\newtheorem{theorem}{Theorem}[section]
\newtheorem{definition}{Definition}[section]
\newtheorem{remark}{Remark}[section]
\newtheorem{lemma}{Lemma}[section]
\newtheorem{proposition}{Proposition}[section]

\newtheorem{example}{Example}
\numberwithin{equation}{section}

\begin{document}
\title[$ p( x )$- laplacian in Sobolev spaces with variable exponents in complete manifolds]{Existence and Uniqueness of Weak solution of $ p( x ) $- laplacian in Sobolev spaces with variable exponents in complete manifolds}
\author[O. Benslimane, A. Aberqi and  J. Bennouna]
{O. Benslimane$^1$, A. Aberqi$^2$ and  J. Bennouna$^1$}
\address{O. Benslimane, A. Aberqi and  J. Bennouna\newline
$^{1}$ Sidi Mohamed Ben Abdellah University, Faculty of Sciences Dhar Al Mahraz, Department of Mathematics, B.P 1796 Atlas Fez, Morocco.\newline
$^{2}$ Sidi Mohamed Ben Abdellah University, National School of Applied Sciences Fez, Morocco.}
\email{$^1$omar.benslimane@usmba.ac.ma}
\email{$^2$aberqi\_ahmed@yahoo.fr}
\email{$^3$jbennouna@hotmail.com}

\maketitle
\textbf{Abstract.} The paper deals with the existence and uniqueness of a non-trivial solution to non-homogeneous $ p ( x ) -$laplacian equations, managed by non polynomial growth operator in the framework of variable exponent Sobolev spaces on Riemannian manifolds. The mountain pass Theorem is used.
\section{Introduction}

Let $ ( M, g ) $ be a complete non-compact Riemannian manifold, we consider the following equation
\begin{equation}\label{1.1}
 - \Delta_{p( x )} u( x ) + h( x, u( x ), \nabla u( x ) ) + | \, u( x ) \,|^{p( x ) - 2} u( x ) = f( x, u( x ) ),
\end{equation} 
where  $-  \Delta_{p( x )} u( x ) = - \,div\,( | \, \nabla u( x ) \, |^{p( x ) - 2} \,.\, \nabla u( x ) ) $ is the $ p( x ) $-laplacian in $( M, g )$. The conditions assumed on the functions $f$ and $h$ are:\\
$ ( f_{1} ): \, f( x, 0 ) = 0 $ and $f$ is measurable to the first variable and continuous to the second variable.\\
$ ( f_{2} ): $ There exists $ \beta \in [ \, p^{-}, p^{+} ] $ such as
$$ 0 < \int_{M} F( x, \alpha ) \,\,dv_{g} ( x ) \leq \int_{M} f( x, \alpha ) \,.\, \frac{\alpha}{\beta} \,\, dv_{g} ( x ) \,\, \mbox{a.e} \,\,x \in M, $$
where $ \displaystyle F( x, \alpha ) = \int_{0}^{\alpha} f( x, t ) \,\, dt $ being the primitive of $ f( x, \alpha )$.\\
$ ( f_{3} ): \, \displaystyle \lim_{| \, \alpha \, | \rightarrow \infty} \frac{f( x, \alpha )}{| \, \alpha \, |^{p( x ) - 1}} = 0 $ uniformly a.e $ x \in M $.\\
\begin{example} $$ f( x, \alpha ) = c \, | \, \alpha \, |^{K( x ) - 1} \,\, \forall c > 0, \,\, \beta \, < \, K( x )\, < p( x ) $$ is a function satisfying the above conditions.
\end{example}
$ ( h_{1} ): \, h( x, s, \xi ) : M \times \mathbb{R} \times \mathbb{R}^{N} \longrightarrow \mathbb{R}^{N} $ be a Carathéodory function such as for a.e $ x \in M $ and for all $ s \in \mathbb{R}$, $ \xi \in \mathbb{R}^{N}$, 
$$ | \, h( x, s, \xi )  \, | \leq \gamma ( x ) + l( s ) \,.\, | \, \xi \, |^{p( x )}, $$
where $ l : \, \mathbb{R} \longrightarrow \mathbb{R}^{+} $ is a continuous increasing positive function that belongs to $ L^{\infty} ( M ) $ and $ \gamma ( x ) \in L^{1} ( M ) $.

\noindent The study of variational problems with non-standard growth conditions is an interesting topic in recent years. $ p( x) -$growth condition can be regarded as an important case of non-standard $ ( p, q )-$growth conditions. Recently, the study of these kind of problems has attracted more and more attention. For example Fan Xian-Ling and Zhang Qi-Hu,  in \cite{fan2003existence}, proved the existence of solutions to the Dirichlet problem of $ p( x )-$laplacian 
$$ \begin{cases}
\, -\mbox{div}  (\, | \, \nabla u \, |^{p( x ) -2} \, \nabla u \, )\, = \, f ( x,  u ) & \text{$x \in \Omega$ }, \\[0.3cm]
\, u \,= \, 0 & \text{$ x \in \partial \Omega$ }, 
\end{cases} $$ 
with several sufficient conditions, and a criterion of existence for an infinite number of pairs of solutions to this problem. For more result we refer the reader to \cite{barnas2012existence, radulescu2015partial}. The typical applications of variable exponent equations include models for electrorheological fluids \cite{acerbi2002regularity}, image restoration processing \cite{hen2006variable}, non-Newtonian fluid dynamics \cite{gwiazda2010monotonicity}, Poisson equation \cite{diening2011lebesgue}, elasticity equations \cite{gwiazda2015thermo, zhikov1997meyer}, and thermistor model \cite{zhikov1997some}.\\
Moving on to another field undergoing great development; the Sobolev space on Riemannian manifolds. The theory of Sobolev space for non compact manifold arose in the 1970s with the work of Aubin, Cantor Hoffman, and Spruck, many of the results presented in their lecture notes have been collected between the 1980s and the 1990s. It has been studied very intensively for over fifty years see \cite{aubin1982nonlinear, hebey2000nonlinear, gadea2012analysis} and also e.g \cite{gala2012new, polidoro2008harnack}.This is also the case for the applications already mentioned to scalar curvature and generalized scalar curvature equation, we quote \cite{benalili2009multiplicity}. Additionally, Yamabe problem for conformal metrics with prescribed scalar curvature \cite{trudinger1968remarks}, and to obtain isoperimetric type inequalities \cite{hebey2000nonlinear}.  \\

Considering that some basic properties of the standard Lebesgue space are not valid in the variable exponent case. For example, Zhikov \cite{zhikov1997some} observed that in general smooth functions are not dense in $ W^{K, p( . )} ( \Omega )$. Besides, the challenges coming are due to the absence of topological properties like convergence and embedding.\\

In this paper, we will be applying the Sobolev spaces with their variable exponents on the non-compact Riemannian manifolds theory to our equation \eqref{1.1}. As for the structure of the paper, we will be sectioning it into: Recalling some definitions and Lemmas. We will then move to prove the existence of non-trivial solution using the mountain pass Theorem, and we will finish it by a demonstration of the uniqueness of non-trivial solution with $ f $ as the Contraction Lipschitz Continuous function.

\section{Framework Space: Notations and Basic Properties}
First of all, we must recall the most important and pertinent properties and notations, by that, referring to \cite{aubin1982nonlinear, gadea2012analysis,  radulescu2015partial} for more details.
\begin{definition}
Let $ \nabla $ be the Levi-Civita connection. If $u$ is a smooth function on $M$, then $ \nabla^{k} u $ denotes the $k-$th covariant derivative of $u$, and $ | \, \nabla^{k} u \, | $ the norm of $ \nabla^{k} u $ defined in local coordinates by
$$ | \, \nabla^{k} u \, |^{2} = g^{i_{1} j_{1}} \cdots g^{i_{k} j_{k}} \, ( \nabla^{k} u )_{i_{1} \cdots i_{k}} \, ( \nabla^{k} u )_{j_{1} \cdots j_{k}} $$
where Einstein's convention is used.\\
$\hspace*{1cm} \bullet $ Given a variable exponent $ p $ in $ \mathcal{P} ( M ) $ and a natural number $k$, introduce 
$$ C^{p(.)}_{k} ( M ) = \{ \, u \in C^{\infty} ( M ) \,\, \mbox{such that } \,\, \forall j \,\, 0 \leq j \leq k \,\, | \, \nabla^{k} u \, | \in L^{p( . ) } \, \} $$
on $ C^{p( . )}_{k} ( M ) $ define the norm 
$$ || \, u \, ||_{L^{p( . )}_{k}} = \sum_{j = 0}^{k} || \, \nabla^{j} u \, ||_{L^{p( . )}} $$
\end{definition}
\begin{definition}
Given $ ( M, g ) $ a smooth Riemannian manifold, and $ \gamma : \, [\, a, \, b \, ] \longrightarrow M $ a curve of class $ C^{1} $, the length of $ \gamma $ is 
$$ l( \gamma ) = \int_{a}^{b} \sqrt{g \, ( \, \frac{d \gamma }{d t }, \, \frac{d \gamma}{d t}\, )} \,\,dt, $$
and for a pair of points $ x, \, y \in M$, we define the distance $ d_{g} ( x, y ) $ between $x$ and $y$ by 
$$ d_{g} ( x, y ) = \inf \, \{ \, l( \gamma ) : \, \gamma: \, [ \, a, \, b \,] \rightarrow M \,\, \mbox{such that} \,\, \gamma ( a ) = x \,\, \mbox{and} \,\, \gamma ( b ) = y \, \} $$
\end{definition}
\begin{definition}
A function $ s: \, M \longrightarrow \mathbb{R} $ is log-Hölder continuous if there exists a constant $c$ such that for every pair of points $ \{ x, \, y \} $ in $ M$ we have
$$ | \, s( x ) - s( y ) \, | \leq \frac{c}{log ( e + \frac{1}{d_{g} ( x, y )} \, )}. $$
We note by $ \mathcal{P}^{log} ( M ) $ the set of log-Hölder continuous variable exponents.
\end{definition}
\begin{proposition}\label{main}
Let $ p \in \mathcal{P}^{log} ( M ) $, and let $ ( \Omega, \phi ) $ be a chart such that 
$$ \frac{1}{2} \delta_{i j } \leq g_{i j} \leq 2 \, \delta_{i j } $$
as bilinear forms, where $ \delta_{i j} $ is the delta Kronecker symbol. Then $ po\phi^{-1} \in \mathcal{P}^{log} ( \phi ( \Omega ) ).$
\end{proposition}
\begin{definition}
We say that the n-manifold $ ( M, g ) $ has property $ B_{vol} ( \lambda, v ) $ if its geometry is bounded in the following sense:\\
$ \hspace*{1cm} \bullet \,\, Rc ( g ) \geq \lambda ( n - 1 ) \, g $ for some $ \lambda $\\
$ \hspace*{1cm} \bullet  $ There exists some $ v > 0 $ such that $ | \, B_{1} ( x ) \, |_{g} \geq v \,\, \forall x \in M.$
\end{definition}
\begin{proposition}\label{main}
Let $ ( M, g ) $ be a complete Riemannian n-manifold. Then, if the embedding $ L^{1}_{1} ( M ) \hookrightarrow L^{\frac{n}{n - 1}} ( M )$ holds, then whenever the real numbers $p$ and $q$ satisfy $$ 1 \leq p < n, $$ and $$ p \leq q \leq p* = \frac{n p}{n - p}, $$ the embedding $ L^{p}_{1} ( M ) \hookrightarrow L^{q} ( M ) $ also holds.
\end{proposition}
\begin{proposition}\label{main}
Assume that the complete n-manifold $ ( M, g ) $ has property $ B_{vol} ( \lambda, v ) $ for some $ ( \lambda, v ).$ Then there exist positive constants $ \delta_{0} = \delta_{0} ( n, \, \lambda, \, v ) $ and $ A = A ( n, \, \lambda, \, v ) $, we have, if $ R \leq \delta_{0} $, if $ x \in M $ if $ 1 \leq p \leq n $, and if $ u \in L^{p}_{1,0} ( \, B_{R} ( x ) \, ) $ the estimate 
$$ || \, u \, ||_{L^{q}} \leq A_{q} \, || \, \nabla u \, ||_{L^{p}},$$ where $ \frac{1}{q} = \frac{1}{p} - \frac{1}{n}.$ 
\end{proposition}
\begin{proposition}\label{main}
Assume that for some $ ( \lambda, v ) $ the complete n-manifold $( M, g ) $ has property $ B_{vol} ( \lambda, v ) $. Let $ q \in \mathcal{P} ( M ) $ be uniformly continuous with $ p^{+} < n.$ Then $ L^{p( . )}_{1} ( M ) \hookrightarrow L^{q( . )} ( M ) \,\, \forall p \in \mathcal{P} ( M ) $ such that $ p \ll q \ll p* = \frac{n p}{n - p}. $ In fact, for $ || \, u \, ||_{L^{p( . )}_{1}} $ sufficiently small we have the estimate $$ \varrho_{q( . )} ( u ) \leq G \, ( \, \varrho_{p( . )} ( u ) + \varrho_{p( . )} ( | \, \nabla u \, | ) \, ), $$ where $G$ is a positive constant depend on $ n, \, \lambda, \, v, \, p $ and $ q $.
\end{proposition}

\section{Existence of Non-trivial Solution}
\begin{definition}
$ u \in L^{p( . )}_{1} ( M ) $ is said to be a weak solution of the problem \eqref{1.1} if for every $ \phi \in D( M ) $ we have
\begin{align}\label{3.1}
\int_{M} &| \, \nabla u ( x ) \, |^{p( x) - 2} g ( \, \nabla u ( x ), \nabla \phi ( x ) ) \,\, dv_{g} ( x ) + \int_{M} h ( x,\, u( x ), \, \nabla u( x ) )\,.\, \phi ( x ) \,\, dv_{g} ( x ) \nonumber \\&+ \int_{M} | \, u( x ) \, |^{p( x ) - 2} u( x ) \, \phi ( x ) \,\, dv_{g} ( x ) = \int_{M} f( x, u( x ) ) \, \phi ( x ) \,\, dv_{g} ( x )
\end{align}
\end{definition}
\begin{theorem}\label{main}
Suppose that $ ( f_{1} ) - ( f_{3} ) $ and $ ( h_{1} ) $ are true. Then the problem \eqref{1.1} possesses a non-trivial weak solution.
\end{theorem}
To access our main premises, the ones shown in the first Theorem, we have to initially demonstrate few lemmas related to the mountain pass Theorem and Palais-Smale condition. Considering the functional
\begin{align*}
A( u ) = & \int_{M} \frac{1}{p( x )} ( \, | \, \nabla u( x ) \, |^{p( x )} + | \, u( x ) \, |^{p( x )} \, ) \,\, dv_{g} ( x ) + \int_{M} H ( x, \,u( x ), \, \nabla u( x ) ) \,\, dv_{g} ( x )\\&  - \int_{M} F( x, \, u( x ) ) \,\, dv_{g} ( x ),
\end{align*}
with $\displaystyle H ( x,\, u( x), \, \nabla u( x) ) = \int_{0}^{\alpha} h ( x, \,t, \, \nabla t ) \,\, dt  $ being the primitive of $ h( x, \, \alpha, \, \nabla \alpha ) $. And $ H( x, 0, 0 ) = 0 $.\\
It's follow from \eqref{3.1}  and the hypothesis of $ ( f_{1} ) - ( f_{3} ), \, ( h_{1} )$ and the above definition of $H$, that $A$ is $ C^{1} $ functional.\\
Let $ D A = D B - D C $ the differential of $ A = B - C $ with 
\begin{align*}
DB =& \int_{M}  \,|\, \nabla u ( x ) \, |^{p( x) - 2} g ( \, \nabla u ( x ), \nabla \phi ( x ) ) \,\, dv_{g} ( x ) \\& + \int_{M} h ( x,\, u( x ), \, \nabla u( x ) ) \,.\, \phi ( x )\,\, dv_{g} ( x ) \\& + \int_{M} | \, u( x ) \, |^{p( x ) - 2} u( x ) \, \phi ( x ) \,\, dv_{g} ( x ) ,
\end{align*} 
 and $$ DC = \int_{M} f( x, u( x ) ) \, \phi ( x ) \,\, dv_{g} ( x ). $$
\begin{lemma}\label{main}
The functional $A$ satisfies mountain pass geometry in the sense that: \\
$ i/ \,\, A( 0 ) = 0. $\\
$ ii/ \,$ There exists $ r, \eta > 0 $ such that $ A( u ) \geq \eta $ if $ || \, u \, || > r.$\\
$ iii/ \, $ There exists $ u, \, || \, u \, || > r $ such that $ A( u ) \leq 0 .$
\end{lemma}
\begin{proof}
$ i/ \, \, A( 0 ) = 0 $ is obvious. \\
$ ii/ \,$ we need the assumptions $ ( f_{2} ), \, ( f_{3} ) $ and $ ( h_{1} ) $, then we obtain
\begin{align*}
\int_{M} F( x, u( x ) ) \,\, dv_{g} ( x )  & \leq \frac{\epsilon}{\beta} \, \int_{M} | \, u( x ) \, |^{p( x )} \,\, dv_{g} ( x ) + \frac{m( \epsilon )}{\beta} \\& \leq \big( \, \frac{ \epsilon + m ( \epsilon )}{\beta} \, \big)  \, \varrho_{p( . )} ( u ),
\end{align*}
and 
\begin{align*}
\bigg| \, \int_{M} H ( x, \, u( x ), \, \nabla u( x ) )  \,\, dv_{g} ( x ) \, \bigg|  &\leq  \int_{M} \bigg( \, \int_{0}^{u} | \, h( x, \, t, \, \nabla t ) \, | \,\, dt \, \bigg) \,\, dv_{g} ( x )  \\ &\leq c_{1} \, \varrho_{p( . )} ( u ) + c_{2} + l( u ) \, \varrho_{p( . )} ( \, | \, \nabla u \, | \, ). 
\end{align*}
Choose $ || \, u \, || = r $ sufficiently small, so that $ \varrho_{p( . )} ( u ) \leq || \, u \, ||_{p( x )}^{p} $ since $ || \, u \, || = r < 1. $ Now using the Poincaré inequality, we get
\begin{align*}
A( u ) = & \int_{M} \frac{1}{p( x )} ( \, | \, \nabla u( x ) \, |^{p( x )} + | \, u( x ) \, |^{p( x )} \, ) \,\, dv_{g} ( x ) + \int_{M} H ( x, \,u( x ), \, \nabla u( x ) ) \,\, dv_{g} ( x )\\& \hspace*{0.2cm} - \int_{M} F( x, \, u( x ) ) \,\, dv_{g} ( x ) \\ & \geq \frac{1}{p^{+}} \, ( \, \varrho_{p( . ) } ( \, | \, \nabla u \, | \, ) + \varrho_{p( . )} ( u ) \, ) - c_{1} \, \varrho_{p( . )} ( u ) - c_{2}  - l( u )\,.\, \varrho_{p( . )} ( \, | \, \nabla u \, | ) \\& \hspace*{0.2cm} - \big( \, \frac{\epsilon + m( \epsilon ) }{\beta} \, \big)\, \varrho_{p( . )} ( u ) \\ & \geq \bigg( \, \frac{1}{p^{+}} - c_{1} - \frac{\epsilon + m( \epsilon )}{\beta} \, \bigg) \, \varrho_{p( . )} ( u ) + \bigg( \, \frac{1}{p^{+}} - l( u ) \, \bigg) \, \varrho_{p( . )} \, ( \,|\, \nabla u \,|\, ) \\ & \geq K_{1} \, \varrho_{p( . )} ( u ) + k_{2} \, \varrho_{p( . )} ( \, | \, \nabla u \, | \, ) \\ & \geq \inf ( K_{1}, \, K_{2} \, ) \, ( \, \varrho_{p( . ) } \, ( u ) + \varrho_{p( . )} \, ( \, | \, \nabla u \, | \, ),
\end{align*}
with $$ K_{1} = \frac{1}{p^{+}} - c_{1} - \frac{\epsilon + m( \epsilon )}{\beta}, \,\,\,\, K_{2} = \frac{1}{p^{+}} - l( u ), $$ and $$ \varrho_{p( . )} ( \, | \, \nabla u \, | \, ) + \varrho_{p( . )} ( u ) \geq 2^{1 - p^{+}} \, || \, u \, ||_{L^{q( . )} ( M ) }^{q^{+} }.$$
Hence, $$ A( u ) \geq \inf ( K_{1}, \, K_{2} ) \, 2^{1 - p^{+}} \, || \, u \, ||_{L^{q( . )} ( M )}^{q^{+}}. $$ So, $$ A( u ) \geq \eta \,\, \mbox{for some} \,\, \eta > 0 .$$

We can prove $ iii/ $ by using $ ( f_{2} ) $ and $ ( h_{1} ), $ for $ t > 0 $ and $ u \neq 0.$ So,
\begin{align*}
 A( tu ) &=  \int_{M} \frac{1}{p( x )} \, ( \, | \, \nabla t u \, |^{p( x )} - | \, t u \, |^{p( x )} \, ) \,\, dv_{g} ( x )  + \int_{M} H ( x, \, t u, \, \nabla t u ) \,\, dv_{g} ( x ) \\& \hspace*{0.2cm}- \int_{M} F( x, tu ) \,\, dv_{g} ( x ) \\& = \int_{M} \frac{t^{p( x )}}{p( x )} \, ( \, | \, \nabla t u \, |^{p( x )} - | \, t u \, |^{p( x )} \, ) \,\, dv_{g} ( x )  + \int_{M} H ( x, \, t u, \, \nabla t u ) \,\, dv_{g} ( x ) \\& \hspace*{0.2cm} -  \int_{M} F( x, tu ) \,\, dv_{g} ( x ) \\& \leq \int_{M} \frac{t^{p( x )}}{p( x )} \, ( \, | \, \nabla t u \, |^{p( x )} - | \, t u \, |^{p( x )} \, ) \,\, dv_{g} ( x ) + c_{1} \, t^{p( x )} \, \int_{M} | \, t u \, |^{p( x )} \,\, dv_{g} ( x ) + c_{2} \\& \hspace*{0.2cm}+ l( t u ) \,. \, t^{p( x )} \, \int_{M} | \, \nabla t u \, |^{p( x )} \,\, dv_{g} ( x ).
\end{align*}
This implies:
\begin{align}\label{3.2}
 A( t u ) &\leq \int_{M} \frac{t^{p^{+} - p^{-}}}{p( x )} \, ( \, | \, \nabla u \, |^{p( x )} - | \, u \, |^{p( x )} \, ) \,\, dv_{g} ( x ) + c_{1} \,.\, t^{p^{-}} \, \int_{M} | \,  u \, |^{p( x )} \,\, dv_{g} ( x ) + c_{2} \nonumber \\ & \hspace*{0.2cm} + l( t u ) \,.\, t^{p^{+}} \, \int_{M} | \, \nabla u \, |^{p( x )} \,\, dv_{g} ( x ),
\end{align}
 dividing \eqref{3.2} by $ t^{p^{+}} $ and passing the limit $ t \longrightarrow \infty $ we get $ A ( t u ) \longrightarrow - \infty $, since $ p^{+} > p^{-}$. \\
 Hence, $ A( u ) $ satisfies the hypothesis of mountain pass Theorem.
\end{proof}
\begin{lemma}\label{main}
The functional $ A $ satisfies Palais-Smale condition.
\end{lemma}
\begin{proof}
 Let $ \{ \, u_{n} \, \} $ be a Palais-Smale sequence, such as the associated sequence of real numbers $ \{ \, A ( u_{n} )\, \} $ is bounded, and $ DA ( u_{n} ) \longrightarrow 0 $ in $ ( \, L^{p( . )}_{1} ( M ) \, )' $. We will first demonstrate that  $ ( \, u_{n} \, ) $ is bounded in $ L^{p( . )}_{1} ( M ).$ And prove it throughout contradiction.\\
Let $ || \, u_{n} \, ||_{L^{p( . )}_{1} ( M )} \longrightarrow \infty $ as $ n \longrightarrow \infty $. Then, we have 
\begin{align*}
A( u_{n} ) - \frac{1}{\beta} < \, DA ( u_{n} ), \, u_{n} \, > \,& = \big( \, \frac{1}{p( x )} - \frac{1}{\beta} \, \big) \, \int_{M} ( \, | \, \nabla u_{n} \, |^{p( x )} - | \, u_{n} \, |^{p( x )} \, ) \,\, dv_{g} ( x ) \\ & \hspace*{0.2cm}+ \int_{M} \big( \,  H( x, u_{n}, \nabla u_{n} ) - h( x, u_{n}, \nabla u_{n} )  \,.\, \frac{u_{n}}{\beta}\, \big) \,\, dv_{g} ( x ) \\& \hspace*{0.2cm} + \int_{M} \big( \, f( x, u_{n} ) \,.\, \frac{u_{n}}{\beta} - F( x, u_{n} ) \, \big) \,\, dv_{g} ( x ),
\end{align*}
by $ ( f_{2} ) $, we obtain
\begin{align*}
A( u_{n} ) - \frac{1}{\beta} < \, DA ( u_{n} ), \, u_{n} \, > \,& \geq \big( \, \frac{1}{p( x )} - \frac{1}{\beta} \, \big) \, \int_{M} ( \, | \, \nabla u_{n} \, |^{p( x )} - | \, u_{n} \, |^{p( x )} \, ) \,\, dv_{g} ( x ) \\ &+ \int_{M} \big( \,  H( x, u_{n}, \nabla u_{n} ) - h( x, u_{n}, \nabla u_{n} )  \,.\, \frac{u_{n}}{\beta}\, \big) \,\, dv_{g} ( x ) \\& \geq J( u_{n} ) + I( u_{n} ),
\end{align*}
with, $$ J( u_{n} ) = \big( \, \frac{1}{p( x )} - \frac{1}{\beta} \, \big) \, \int_{M} ( \, | \, \nabla u_{n} \, |^{p( x )} - | \, u_{n} \, |^{p( x )} \, ) \,\, dv_{g} ( x ), $$
and $$ I ( u_{n} ) = \int_{M} \big( \,  H( x, u_{n}, \nabla u_{n} ) - h( x, u_{n}, \nabla u_{n} )  \,.\, \frac{u_{n}}{\beta}\, \big) \,\, dv_{g} ( x ). $$
\begin{align*}
J( u_{n} ) \geq \big( \, \frac{1}{p^{+}} - \frac{1}{\beta} \, \big) \, ( \, \varrho_{p( . )} ( u_{n} ) + \varrho_{p( . )} ( \, |\, \nabla u_{n} \, | \, ) \, ),
\end{align*}
or  $ \varrho_{p( . )} ( \, | \, \nabla u_{n} \, | \, ) + \varrho_{p( . )} ( u_{n} ) \geq 2^{1 - p^{+}} \, || \, u_{n} \, ||_{L^{q( . )} ( M )}^{q^{+}}.$ So,
\begin{align*}
J ( u_{n} ) & \geq \big( \, \frac{1}{p^{+}} - \frac{1}{\beta} \, \big) \, 2^{1 - p^{+}} \, ||\, u_{n} \, ||^{q^{+}}_{L^{q( . )} ( M )} \\& \geq \big( \, \frac{1}{p^{+}} - \frac{1}{\beta} \, \big) \, .\, \frac{2^{1 - p^{+}} }{G^{q^{+}}} \,.\, || \, u_{n} \, ||_{L^{p( . )}_{1} ( M )}^{p^{+}},
\end{align*}
with $ G $ being the positive constant of the embedding $ L^{p( . )}_{1} ( M ) \hookrightarrow L^{q( . )} ( M ) $. And by $ ( h_{1} ) $, we get
\begin{align*}
\int_{M} h ( x, u_{n}, \nabla u_{n} ) \, .\, u_{n} \,\, dv_{g} ( x ) & \leq \int_{M} ( \, \gamma ( x ) u_{n} + l( u_{n} ) \,.\, | \, \nabla u_{n} \, |^{p( x )} \, .\, u_{n} \, ) \,\, dv_{g} ( x ) \\& \leq c_{1} \, || \, u_{n} \, ||_{L^{p( . )}_{1} ( M )} + l( u_{n} ) \, ||\, | \, \nabla u_{n} \,|\, ||_{L^{p( . )}_{1} ( M )}^{p( x )} \, || \, u_{n} \, ||_{L^{p( . )}_{1} ( M )}.
\end{align*}
And
\begin{align*}
\int_{M} H ( x, u_{n}, \nabla u_{n} ) \,\, dv_{g} ( x) & = \int_{M} \, \bigg( \, \int_{0}^{u_{n}} f( x, t, \nabla t ) \,\, dt \, \bigg) \,\, dv_{g} ( x )\\ & \geq - c_{1} \, \varrho_{p( . )} ( u_{n} ) - c_{2} - l( u_{n} ) \,.\, \varrho_{p( . )} ( | \, \nabla u_{n} \, | ).
\end{align*}
Thus, 
\begin{align*}
I ( u_{n} ) & \geq  c_{1} \, \varrho_{p( . )} ( u_{n} ) - c_{2} - l( u_{n} ) \,.\, \varrho_{p( . )} ( | \, \nabla u_{n} \, | ) \\& \hspace*{0.3cm}- \frac{c_{1}}{\beta} \,.\, ||\, u_{n} \,||_{L^{p( . )}_{1} ( M )} - \frac{l( u_{n} )}{\beta} \,.\, ||\, u_{n}\,||_{L^{p( . )}_{1} ( M )} \,.\, \varrho_{p( . )} ( |\, \nabla u_{n} \, | )\\ & \geq - \bigg( \, c_{1} \,.\, || \, u_{n} \, ||_{L^{p( . )}_{1} ( M )}^{p( x ) - 1} - \frac{c_{1}}{\beta} \, \bigg) \,.\, ||\, u_{n} \, ||_{L^{p( . )}_{1} ( M )} - c_{2} \\& \hspace*{0.3cm} - \bigg( \, l( u_{n} ) + \frac{l( u_{n} )}{\beta} \,.\, ||\, u_{n} \, ||_{L^{p( . )}_{1} ( M )} \, \bigg) \,.\, \varrho_{p( . )} ( |\, \nabla u_{n} \,| ).
\end{align*}
Hence, 
\begin{align}\label{3.3}
A( u_{n} ) - \frac{1}{\beta} < \, DA ( u_{n} ), \, u_{n} \, > \,& \geq \big( \, \frac{1}{p^{+}} - \frac{1}{\beta} \, ) \,.\, 2^{1 - p^{+}} \, .\, || \, u_{n} \, ||_{L^{p( . )}_{1} ( M )}^{p^{+}} \nonumber \\ & \hspace*{0.2cm} - \bigg( \, c_{1} \,.\, || \, u_{n} \, ||_{L^{p( . )}_{1} ( M )}^{p^{+} - 1} - \frac{c_{1}}{\beta} \, \bigg) \,.\, ||\, u_{n} \, ||_{L^{p( . )}_{1} ( M )} - c_{2} \nonumber \\ & \hspace*{0.2cm}- \bigg( \, l( u_{n} ) + \frac{l( u_{n} )}{\beta} \,.\, ||\, u_{n} \, ||_{L^{p( . )}_{1} ( M )} \, \bigg) \,.\, \varrho_{p( . )} ( |\, \nabla u_{n} \,| ).
\end{align}
Now, dividing both sides of \eqref{3.3} by $ || \, u_{n} \, ||_{L^{p( . )}_{1} ( M )} $ and passing to the limit $ n \longrightarrow \infty $, we get $ 0 \geq \infty $ as $ p^{+} > 1$ which is absurd.\\
Hence, $ ( u_{n} ) $ is bounded in $ L^{p( . )}_{1} ( M ).$\\
$\bullet $ Since $\{ \, u_{n} \, \} $ is bounded in $ L^{p( . )}_{1} ( M )$, there exists a subsequence of $ \{ \, u_{n} \, \},$ noted again by $ \{ \, u_{n} \, \} ,$ that converges in $ L^{p( . )}_{1} ( M )$. We will prove that $ \{ \, u_{n} \, \} $ is Cauchy in $ L^{p( . )}_{1} ( M )$ i.e that 
\begin{equation}\label{3.4}
\lim_{m, n \rightarrow \infty} \varrho_{p( . )} ( u_{m} - u_{n} ) = 0 = \lim_{m, n \rightarrow \infty} \varrho_{p( . )} ( | \, \nabla ( \, u_{m} - u_{n} ) \, | \, ).
\end{equation} 
And to finish the demonstration of the Theorem 1, it suffice to show that the subsequence of $ \{ \, u_{n} \, \} $ still noted by $ \{ \, u_{m} \, \} $ is a Cauchy sequence in $ L^{p( . )}_{1} ( M ).$\\
Consider the following functionals
\begin{equation}\label{3.5}
B_{1} ( u ) = \int_{p( x ) < 2} ( \, | \, \nabla u( x ) \, |^{p( x )} + | \, u( x ) \, |^{p( x )} + H ( x, u, \nabla u ) \, ) \,\, dv_{g} ( x ),
\end{equation}
and
\begin{equation}\label{3.6}
B_{2} ( u ) = \int_{p( x ) \geq 2} ( \, | \, \nabla u( x ) \, |^{p( x )} + | \, u( x ) \, |^{p( x )} + H ( x, u, \nabla u ) \, ) \,\, dv_{g} ( x ),
\end{equation}
on $ L^{p( . )}_{1} ( M ),$ and note that
\begin{equation}\label{3.7}
\varrho_{p( . )} ( u_{m} - u_{n} ) + \varrho_{p( . )} ( \, | \, \nabla ( \, u_{m} - u_{n} \, ) \, | \, ) = B_{1} ( u_{m} - u_{n} ) - B_{2} ( u_{m} - u_{n} ),
\end{equation}
consider also the following inequalities
\begin{equation}\label{3.8}
| \, \zeta - \mu \, |^{p} \leq \frac{2}{p - 1} \, \big[ \, ( \, | \, \zeta \, |^{p - 2} \, \zeta - | \, \mu \,|^{p - 2} \, \mu \, ) \,.\, ( \, \zeta - \mu \, ) \, \big]^{\frac{p}{2}} \,.\, ( \, | \, \zeta \, | - | \, \mu \, | \, )^{\frac{2q - q^{2}}{2}} \,\,\,\, \mbox{for} \,\, 1 \leq p \leq 2,
\end{equation}
and 
\begin{equation}\label{3.9}
| \, \zeta - \mu \, |^{p} \leq 2^{p} \,( \, | \, \zeta \, |^{p - 2} \, \zeta - | \, \mu \,|^{p - 2} \, \mu \, ) \,.\, ( \, \zeta - \mu \, ) \,\,\, \mbox{for} \,\,\, p \geq 2,
\end{equation}
where, $ \zeta, \, \mu \in \mathbb{R}^{N} $ and \eqref{3.1} we deduce that
\begin{align*}
B_{2} ( u_{m} - u_{n} )  &\leq \int_{p( x ) \geq 2} ( \, | \, \nabla ( \, u_{m} - u_{n} \, ) \, |^{p( x )} + | \, u_{m} - u_{n} \, |^{p( x )} + H ( x, u_{m}, \nabla u_{m} )\\& \hspace*{0.2cm}- H ( x, u_{n}, \nabla u_{n} ) \, ) \,\, dv_{g} ( x ) \\& \leq 2^{p} \, \bigg[ \, \int_{p \geq 2} ( \, | \, \nabla u_{m} \, |^{p - 2} \, \nabla u_{m} - | \, \nabla u_{n} \, |^{p - 2} \, \nabla u_{n} \, ) \, .\, \nabla ( \, u_{m} - u_{n} ) \,\, dv_{g} ( x ) \\& \hspace*{0.2cm}+ \int_{p \geq 2} ( \, | \, u_{m} \, |^{p - 2} \, u_{m} - | \, u_{n} \, |^{p - 2}  \, u_{n} \, ) \,.\, ( u_{m} - u_{n} ) \,\, dv_{g} ( x ) \, \bigg] \\& \hspace*{0.2cm}+ c_{1} \, \int_{p \geq 2} | \, u_{m} \, |^{p} \,\, dv_{g} ( x ) + c_{2} \\ & \hspace*{0.2cm}+ l( u_{m} ) \, \int_{p \geq 2} | \, \nabla u_{m} \, |^{p} \,\, dv_{g} ( x ) - c_{1} \, \int_{p \geq 2} | \, u_{n} \, |^{p} \,\, dv_{g} ( x ) - c_{3}\\& \hspace*{0.2cm}- l( u_{n} ) \, \int_{p \geq 2} | \, \nabla u_{n} \, |^{p} \,\, dv_{g} ( x ).
\end{align*}
Thus, 
\begin{equation}\label{3.10}
2^{-p^{+}} \, B_{2} ( u_{m} - u_{n} ) \leq < \, u_{m} - u_{n}, \, DB( u_{m} ) - DB( u_{n} ) \, >.
\end{equation}
And using \eqref{3.5} and \eqref{3.8}, we obtain
\begin{align}\label{3.11}
\frac{p^{-} - 1}{2} \, B_{1} ( u_{m} - u_{n} )  &\leq \int_{p < 2} g ( \, | \, \nabla u_{m} \, |^{p - 2} \nabla u_{m} - | \, \nabla u_{n} \, |^{p - 2} \, \nabla u_{n}, \, \nabla ( u_{m} - u_{n} )  \, )\nonumber\\ & \hspace*{0.2cm} \times \,( |\, \nabla u_{m} \, | + | \, \nabla u_{n} \, | )^{\frac{2q - q^{2}}{2}} \,\, dv_{g} ( x ) \nonumber \\ & \hspace*{0.2cm} + \int_{p < 2} ( \, | \, \nabla u_{m} \, |^{p - 2} u_{m} - | \, u_{n} \, |^{p - 2} \, u_{n} \, ) \,.\, ( u_{m} - u_{n} \, )^{\frac{p}{2}} \nonumber \\ & \hspace*{0.2cm} \times \,( |\, u_{m} \, | + | \, u_{n} \, | )^{\frac{2q - q^{2}}{2}} \,\, dv_{g} ( x ) \nonumber \\ & \hspace*{0.2cm} + \int_{p < 2} H ( x, u_{m}, \nabla u_{m} ) \,\, dv_{g} ( x ) - \int_{p < 2} H ( x, u_{n}, \nabla u_{n} ) \,\, dv_{g} ( x ).
\end{align}
\begin{remark}\label{main}
If $a$ and $b$ are two positive functions on $M$, then by Hölder's inequality
\begin{equation}\label{3.12}
\int_{p < 2} a^{\frac{p}{2}} \, b^{\frac{2p - p^{2}}{2}} \leq 2 \, || \, \mathbb{1}_{p < 2} \,\, a^{\frac{p}{2}} \, ||_{L^{\frac{2}{p}}} \, .\, || \, \mathbb{1}_{p < 2} \,\, b^{\frac{2p - p^{2}}{2}} \, ||_{L^{\frac{2}{ 2 - p}}}.
\end{equation}
where $ \mathbb{1} $ is the indicator function of $ M$, moreover, since 
$$ ||\, \mathbb{1}_{p < 2} \,\,a^{\frac{2}{p}} \, ||_{L^{\frac{2}{p}}} \leq \max \{ \, \varrho_{1} ( a ), \, \varrho_{1} ( a )^{\frac{p^{-}}{2}} \, \} $$
and 
$$ || \, \mathbb{1}_{p < 2} \,\, b^{\frac{2p - p^{2}}{2}} \, ||_{L^{\frac{2}{ 2 - p}}} \leq \max \{ \, \varrho_{p} ( b )^{\frac{2 - p^{-}}{2}}, \, 1 \, \},$$
\end{remark}
we get,
\begin{equation}\label{3.13}
\int_{p < 2} a^{\frac{p}{2}} \, .\, b^{\frac{2p - p^{2}}{2}} \leq 2 \, \max \, \{ \, \varrho_{1} ( a ), \varrho_{1} ( a )^{\frac{p^{-}}{2}} \, \} \, \max \, \{ \, \varrho_{p} ( b )^{\frac{2p - p^{2}}{2}}, \, 1 \, \}.
\end{equation}
Thus, using \eqref{3.13} twice in \eqref{3.11} we infer that
\begin{equation}\label{3.14}
( \, P^{-} - 1 \, ) \, B_{1} ( u_{m} - u_{n} ) \leq P ( \, < \, u_{m} - u_{n}, DB( u_{m} ) - DB( u_{n} ) \, > \, ) \, Q ( \, ||\, u_{m} \, ||_{L^{p}_{1}}, \, || \, u_{n} \, ||_{L^{p}_{1}} \, ),
\end{equation}
for some constructible continuous functions $ P( x ) $ and $ Q ( y, z ) $ with $ P ( 0 ) = 0$.\\
Since the sequence $ \{ \, ||\, u_{n} \, ||_{L^{p( . )}_{1}} \, \} $ is bounded, from \eqref{3.14}, \eqref{3.10} and \eqref{3.7} we conclude that if $$ \lim_{m, n \rightarrow \infty} < \, u_{m} - u_{n}, \, DB( u_{m} ) - DB( u_{n} ) \, > = 0. $$
To obtain such that a limit, recall that $ DA = DB - DC $ hence
\begin{align*}
< \, u_{m} - u_{n}, DB ( u_{m} ) - DB ( u_{n} ) \, > \, =& \, < \, u_{m} - u_{n}, \, DA ( u_{m} ) - DA ( u_{n} ) \, >\\& + < \, u_{m} - u_{n}, \, DC ( u_{m} ) - DC ( u_{n} ) \, >,
\end{align*}
we estimate 
\begin{align*}
< \, u_{m} - u_{n}, \, DC ( u_{m} ) - DC ( u_{n} ) \, > \leq & \frac{1}{G} \, || \, N_{f} ( u_{m} ) - N_{f} ( u_{n} ) \, ||_{( \, L^{p( . )}_{1} ( M ) \, )'} \\& \times ||\, u_{m} - u_{n} \, ||_{L^{q( . )}_{1} ( M )},
\end{align*}
where $ G$ being the constant of the embedding $ L^{p( . )}_{1} ( M ) \hookrightarrow L^{q( . )} ( M ) $ \\ and $ N_{f} : \, L^{p( . )} ( M ) \longrightarrow ( \, L^{p( . )} ( M ) \, )' $ is the Nemytskii operator induced by $f$, with \\$ N_{f} [ u ] ( x ) = f ( x, u( x ) ) $.\\
Hence, using the continuity  of Nemytskii map, the convergence of $ \{ \, u_{n} \, \} $ in $ L^{p( . )} ( M ) $ the boundedness of $ \{ \, || \, u_{n} \, ||_{L^{q( . )} ( M )} \, \} $ and the hypothesis on $ DA ( u_{n} ) \longrightarrow \infty $ as $ n \longrightarrow \infty $, we deduce that the functional $A$ satisfies the Palais-Smale Condition.
\end{proof}

\section{Uniqueness of Non-trivial Solution} 
Now, let us show the uniqueness of solution.
\begin{theorem}\label{main}
Assume that condition $ ( f_{1} ) - ( f_{3} ) $ and $ ( h_{1} ) $ are true. With $ f \in  L^{p'( x )} ( M  )$ is a contraction Lipschitz continuous function such that $ \exists \alpha $ with $ 0 \leq \alpha < 1 $ 
\begin{equation} \label{4.1}
|\, f(\, x, u_{1} \, ) - f ( \, x, u_{2} \, ) \, | \leq \alpha \, | \, u_{1} - u_{2} \, |.
\end{equation}
Then, the problem \eqref{1.1} has a unique non-trivial solution.
\end{theorem} 
\begin{proof}
 Suppose that $ u_{1} $ and $ u_{2} $ satisfy \eqref{1.1}. Then we have 
\begin{align*}
\int_{M} | \, &\nabla u_{1} \, |^{p( x ) - 2} \,.\, \nabla u_{1} \,.\, ( \, \nabla  u_{1} - \nabla u_{2} \, ) \,\, dv_{g} ( x ) + \int_{M} h ( \, x, u_{1}, \nabla u_{1} \, ) \,.\, ( \, u_{1} - u_{2} \, ) \,\, dv_{g} ( x )\\& + \int_{M} | \, u_{1} \, |^{p( x ) - 2} \,.\, u_{1} \,.\, ( \, u_{1} - u_{2} \, ) \,\, dv_{g} ( x ) = \int_{M} f( \, x, u_{1} \, ) \,.\, ( \, u_{1} - u_{2} \, ) \,\, dv_{g} ( x ),
\end{align*}
and 
\begin{align*}
\int_{M} | \, &\nabla u_{2} \, |^{p( x ) - 2} \,.\, \nabla u_{2} \,.\, ( \, \nabla  u_{2} - \nabla u_{1} \, ) \,\, dv_{g} ( x ) + \int_{M} h ( \, x, u_{2}, \nabla u_{2} \, ) \,.\, ( \, u_{2} - u_{1} \, ) \,\, dv_{g} ( x )\\& + \int_{M} | \, u_{2} \, |^{p( x ) - 2} \,.\, u_{2} \,.\, ( \, u_{2} - u_{1} \, ) \,\, dv_{g} ( x ) = \int_{M} f( \, x, u_{2} \, ) \,.\, ( \, u_{2} - u_{1} \, ) \,\, dv_{g} ( x ).
\end{align*}
Adding the above two equations yields.
\begin{align*}
\int_{M} & \bigg( \, |\, \nabla u_{1} \,|^{p( x ) - 2} \,.\, \nabla u_{1} - | \, \nabla u_{2} \,|^{p( x ) - 2} \,.\, \nabla u_{2} \, \bigg) \,.\, ( \, \nabla u_{1} - \nabla u_{2} \, ) \,\, dv_{g} ( x ) \\& + \, \int_{M} \bigg( \, h ( \, x, u_{1}, \nabla u_{1} \, ) - h ( \, x, u_{2}, \nabla u_{2} \, ) \, \bigg) \,.\, ( \, u_{1} - u_{2} \, ) \,\, dv_{g} ( x ) \\ &+ \int_{M} \bigg( \, |\, u_{1} \,|^{p( x ) - 2} \,.\, u_{1} - | \, u_{2} \,|^{p( x ) - 2} \,.\, u_{2} \, \bigg) \,.\, ( \, u_{1} - u_{2} \, ) \,\, dv_{g} ( x ) \\ & = \int_{M} ( \, f ( \, x, u_{1} \, ) - f ( \, x, u_{2} \, ) \, ) \,.\, ( \, u_{1} - u_{2} \, ) \,\, dv_{g} ( x ),
\end{align*}
by $ ( f_{2} ), \, ( f_{3} ) $ and \eqref{4.1} we obtain
\begin{align*}
\int_{M} & \bigg( \, |\, \nabla u_{1} \,|^{p( x ) - 2} \,.\, \nabla u_{1} - | \, \nabla u_{2} \,|^{p( x ) - 2} \,.\, \nabla u_{2} \, \bigg) \,.\, ( \, \nabla u_{1} - \nabla u_{2} \, ) \,\, dv_{g} ( x ) \\& + \, \int_{M} \bigg( \, h ( \, x, u_{1}, \nabla u_{1} \, ) - h ( \, x, u_{2}, \nabla u_{2} \, ) \, \bigg) \,.\, ( \, u_{1} - u_{2} \, ) \,\, dv_{g} ( x ) \\ &+ \int_{M} \bigg( \, |\, u_{1} \,|^{p( x ) - 2} \,.\, u_{1} - | \, u_{2} \,|^{p( x ) - 2} \,.\, u_{2} \, + \alpha \, ( \, u_{1} - u_{2} \,) \, \bigg) \,.\, ( \, u_{1} - u_{2} \, ) \,\, dv_{g} ( x ) \geq 0. 
\end{align*}
Then,
\begin{align*}
\int_{M}& \bigg( \, |\, h ( \, x, u_{1}, \nabla u_{1} \, )\,| - |\, h ( \, x, u_{2}, \nabla u_{2} \, )\,| \, \bigg) \,.\, ( \, u_{1} - u_{2} \, ) \,\, dv_{g} ( x ) \geq\\& \int_{M}  \bigg( \, |\, \nabla u_{1} \,|^{p( x ) - 2} \,.\, |\, \nabla u_{1} \,| - | \, \nabla u_{2} \,|^{p( x ) - 2} \,.\, | \, \nabla u_{2} \,| \, \bigg) \,.\, ( \, | \, \nabla u_{1} \, | - | \,\nabla u_{2} \,| \, ) \,\, dv_{g} ( x ) \\ &+ \int_{M} \bigg( \, |\, u_{1} \,|^{p( x ) - 2} \,.\,| \, u_{1} \, | - | \, u_{2} \,|^{p( x ) - 2} \,.\, |\, u_{2}\,| \, \bigg) \,.\, ( \,|\, u_{1} \,| - |\,u_{2}\,| \, ) \,\, dv_{g} ( x ) \\ & \geq \int_{M} \big( \, | \, \nabla u_{1} \,|^{p( x ) - 1} - |\, \nabla u_{2} \,|^{p( x ) - 1} \, \big) \,.\, ( \, |\, \nabla u_{1} \,| - |\, \nabla u_{2} \,| \, ) \,\, dv_{g} ( x ) \\ & \hspace*{0.3cm}+ \int_{M} ( \, |\, u_{1} \,|^{p( x ) - 1} - |\, u_{2} \, |^{p( x ) - 1} \, ) \,.\, ( \, |\, u_{1} \,| - | \, u_{2}\,| \, ) \,\, dv_{g} ( x ) \\& \hspace*{0.3cm}+ \alpha \, \int_{M} ( \, | \, u_{1} \, | -  | \, u_{2} \,| \, )^{2} \,\, dv_{g} ( x ).
\end{align*}
So, by $ ( h_{1} ) $ we have 
\begin{align*}
\int_{M}& \bigg( \, |\, h ( \, x, u_{1}, \nabla u_{1} \, )\,| - |\, h ( \, x, u_{2}, \nabla u_{2} \, )\,| \, \bigg) \,.\, ( \, u_{1} - u_{2} \, ) \,\, dv_{g} ( x ) \\ & \leq \int_{M} ( \, l( u_{1} ) \,.\, |\, \nabla u_{1} \, |^{p( x )} - l( u_{2} ) \,.\, |\, \nabla u_{2} \,|^{p( x )} \, )  \,.\, ( \, |\, u_{1} \, | - | \, u_{2} \,| \, ) \,\, dv_{g} ( x ).
\end{align*}
Hence,
\begin{align*}
\int_{M} & ( \, l( u_{1} ) \,.\, |\, \nabla u_{1} \, |^{p( x )} - l( u_{2} ) \,.\, |\, \nabla u_{2} \,|^{p( x )} \, )  \,.\, ( \, |\, u_{1} \, | - | \, u_{2} \,| \, ) \,\, dv_{g} ( x ) \\ & \geq \int_{M} \big( \, | \, \nabla u_{1} \,|^{p( x ) - 1} - |\, \nabla u_{2} \,|^{p( x ) - 1} \, \big) \,.\, ( \, |\, \nabla u_{1} \,| - |\, \nabla u_{2} \,| \, ) \,\, dv_{g} ( x ) \\ & \hspace*{0.3cm}+ \int_{M} ( \, |\, u_{1} \,|^{p( x ) - 1} - |\, u_{2} \, |^{p( x ) - 1} \, ) \,.\, ( \, |\, u_{1} \,| - | \, u_{2}\,| \, ) \,\, dv_{g} ( x ) \\& \hspace*{0.3cm}+ \alpha \, \int_{M} ( \, | \, u_{1} \, | -  | \, u_{2} \,| \, )^{2} \,\, dv_{g} ( x ).
\end{align*}
Which implies that
\begin{align*}
0 \geq & \int_{M} \big( \, | \, \nabla u_{1} \,|^{p( x ) - 1} - |\, \nabla u_{2} \,|^{p( x ) - 1} \, \big) \,.\, ( \, |\, \nabla u_{1} \,| - |\, \nabla u_{2} \,| \, ) \,\, dv_{g} ( x ) \\ & + \int_{M} ( \, |\, u_{1} \,|^{p( x ) - 1} - l( u_{1} )\,.\, |\, \nabla u_{1} \,|^{p( x )} - | \, u_{2} \,|^{p( x ) - 1} - l( u_{2} ) \,.\, | \, \nabla u_{2} \,|^{p( x )} \, ) \\& \hspace*{0.2cm} \times ( \, |\, u_{1}\,| - |\,u_{2}\,| \, )\,\, dv_{g} ( x ) \\& + \alpha \, \int_{M} ( \, | \, u_{1} \, | -  | \, u_{2} \,| \, )^{2} \,\, dv_{g} ( x )  \\ & \geq \int_{M} \big( \, | \, \nabla u_{1} \,|^{p( x ) - 1} - |\, \nabla u_{2} \,|^{p( x ) - 1} \, \big) \,.\, ( \, |\, \nabla u_{1} \,| - |\, \nabla u_{2} \,| \, ) \,\, dv_{g} ( x ).
\end{align*}
Hence, we obtain that $ u_{1} = u_{2} $ almost everywhere. This complete the proof of uniqueness.
\end{proof}

\section{Conclusion}

In sum, the first and second Theorems, are valid and substantiated. We can then say that the solution of the problem \eqref{1.1} is a unique and non-trivial one.


\end{document}